\documentclass[12pt]{article}
\usepackage{amssymb}
\usepackage{amsthm}
\usepackage{amsmath}
\usepackage{graphicx}
\usepackage{enumerate}
\usepackage{pict2e}

\DeclareMathOperator{\Min}{Min}

\newtheorem{theorem}{Theorem}[section]
\newtheorem{definition}[theorem]{Definition}
\newtheorem{lemma}[theorem]{Lemma}
\newtheorem{proposition}[theorem]{Proposition}

\newtheorem{example}[theorem]{Example}
\newtheorem{corollary}[theorem]{Corollary}

\title{Orthomodular and generalized orthomodular posets}
\author{Ivan~Chajda, Miroslav~Kola\v r\'ik and Helmut~L\"anger}
\date{}

\begin{document}

\footnotetext{Support of the research of the first and third author by the Austrian Science Fund (FWF), project I~4579-N, and the Czech Science Foundation (GA\v CR), project 20-09869L, entitled ``The many facets of orthomodularity'', is gratefully acknowledged.}

\maketitle

\begin{abstract}
We prove that the 18-element non-lattice orthomodular poset depicted in the paper is the smallest one and unique up to isomorphism. Since not every Boolean poset is orthomodular, we consider the class of the so-called generalized orthomodular posets introduced by the first and third author in a previous paper. We show that this class contains all Boolean posets and we study its subclass consisting of horizontal sums of Boolean posets. For this purpose we introduce the concept of a compatibility relation and the so-called commutator of two elements. We show the relationship between these concepts and we introduce the notion of a ternary discriminator for these posets. Numerous examples illuminating these concepts and results are included in the paper.
\end{abstract}



\section{Introduction}

It is well-known that the set of closed subspaces of a Hilbert space forms a complete orthomodular lattice with respect to set-inclusion. Because these subspaces correspond to self-adjoint bounded operators which correspond to observables in quantum measurements, this orthomodular lattice is often considered as an algebraic counterpart of the logic of quantum mechanics, see e.g.\ \cite{BV} or \cite H. Recall that an {\em ortholattice} is a bounded lattice $(L,\vee,\wedge,{}',0,1)$ with an antitone involution $'$ which is a complementation, and an {\em orthomodular lattice} is an ortholattice $(L,\vee,\wedge,{}',0,1)$ satisfying the so-called {\em orthomodular law}, i.e.
\begin{enumerate}
\item[(OM)] if $x\leq y$ then $y=x\vee(y\wedge x')$
\end{enumerate}
which is equivalent to its dual
\begin{enumerate}
\item[] if $x\leq y$ then $x=y\wedge(x\vee y')$.
\end{enumerate}
However, it was recognized later that if the elements $x$ and $y$ are not {\em orthogonal}, i.e.\ if not $x\leq y'$ then the join $x\vee y$ need not exist in accordance with quantum theory. Hence, so-called orthomodular posets were introduced (see e.g.\ \cite F) as follows:

An {\em orthomodular poset} is a bounded poset $(P,\leq,{}',0,1)$ with an antitone involution~$'$ which is a complementation satisfying the following conditions:
\begin{enumerate}
\item[(i)] if $x\perp y$ then $x\vee y$ is defined,
\item[(ii)] if $x\leq y$ then $y=x\vee(y\wedge x')$.\quad(OM)
\end{enumerate}
Here and in the following, $x\perp y$ means $x\leq y'$. Observe that the expression in (ii) is well-defined because $x\perp y'$ yields that $y'\vee x$ exists and, by De Morgan's laws, also $y\wedge x'=(y'\vee x)'$ is defined and, due to $y\wedge x'\perp x$ also $x\vee(y\wedge x')$ is defined. Of course, (ii) is equivalent to its dual
\[
x\leq y\text{ implies }x=y\wedge(x\vee y').
\]
It is evident that if the lattice $\mathbf L=(L,\vee,\wedge,{}',0,1)$ is Boolean, i.e.\ a distributive complemented lattice then it is orthomodular. Unfortunately, a similar result does not hold for distributive posets. This is the reason why we introduced the concept of a generalized orthomodular poset (see e.g.\ \cite{CL22}) which, as we will show, can be also a Boolean poset. Hence, we essentially extend the class of orthomodular posets in such a way that they share more natural properties with orthomodular lattices than orthomodular posets do. This is one of our goals in this paper. The second problem connected with orthomodular posets was to find such a poset of minimal size that is not a lattice. As far as we know, this problem was not solved previously. We will present a complete solution.

\section{Basic concepts}

In the following we need several concepts and notations which are presented in this section.

Let $\mathbf P=(P,\leq)$ be a poset, $A,B\subseteq P$ and $a,b\in P$. We define $A\leq B$ if and only if $x\leq y$ for all $x\in A$ and all $y\in B$. Instead of $A\leq\{b\}$, $\{a\}\leq B$ and $\{a\}\leq\{b\}$ we simply write $A\leq b$, $a\leq B$ and $a\leq b$, respectively. The sets
\begin{align*}
L(A) & :=\{x\in P\mid x\leq A\}, \\
U(A) & :=\{x\in P\mid A\leq x\}
\end{align*}
are called the {\em lower cone} and {\em upper cone} of $A$, respectively. Instead of $L(A\cup B)$, $L(A\cup\{b\})$, $L(\{a,b\})$ and $L\big(U(A)\big)$ we write $L(A,B)$, $L(A,b)$, $L(a,b)$ and $LU(A)$, respectively. Analogously, we proceed in similar cases. Recall that $\mathbf P$ is called {\em distributive} (see e.g.\ \cite{LR}) if it satisfies the identity
\[
L\big(U(x,y),z\big)\approx LU\big(L(x,z),L(y,z)\big)
\]
or, equivalently, one of the following identities:
\begin{align*}
UL\big(U(x,y),z\big) & \approx U\big(L(x,z),L(y,z)\big), \\
 U\big(L(x,y),z\big) & \approx UL\big(U(x,z),U(y,z)\big), \\
LU\big(L(x,y),z\big) & \approx L\big(U(x,z),U(y,z)\big).
\end{align*}
Here and in the following $L\big(U(x,y),z\big)\approx LU\big(L(x,z),L(y,z)\big)$ means that
\[
L\big(U(x,y),z\big)=LU\big(L(x,z),L(y,z)\big)\text{ holds for all }x,y,z\in P.
\]
It can be easily seen that if $\mathbf P$ is a lattice then it is a distributive poset if and only if it satisfies the distributive law
\[
(x\vee y)\wedge z\approx(x\wedge z)\vee(y\wedge z).
\]

\begin{lemma}\label{lem4}
Let $(P,\leq,0,1)$ be a bounded distributive poset and $a,b,c,a',b'\in P$. Then {\rm(i)} and {\rm(ii)} hold:
\begin{enumerate}[{\rm(i)}]
\item If $b$ and $c$ are complements of $a$ then $b=c$,
\item if $a\leq b$, $U(a,a')=\{1\}$ and $L(b,b')=\{0\}$ then $b'\leq a'$.
\end{enumerate}
\end{lemma}

\begin{proof}
\
\begin{enumerate}[(i)]
\item We have
\begin{align*}
L(b) & =L(1,b)=L\big(U(c,a),b\big)=LU\big(L(c,b),L(a,b)\big)=LU\big(L(c,b),0\big)= \\
     & =LU\big(L(b,c),0)=LU\big(L(b,c),L(a,c)\big)=L\big(U(b,a),c\big)=L(1,c)=L(c)
\end{align*}
and hence $b=c$.
\item We have
\[
\{0\}\subseteq L(a,b')\subseteq L(b,b')=\{0\}
\]
and hence $L(a,b')=\{0\}$ which implies
\begin{align*}
b' & \in L(b')=L(1,b')=L\big(U(a,a'),b'\big)=LU\big(L(a,b'),L(a',b')\big)= \\
   & =LU\big(0,L(a',b')\big)=LUL(a',b')=L(a',b')\subseteq L(a'),
\end{align*}
i.e.\ $b'\leq a'$.
\end{enumerate}
\end{proof}

In order to avoid the mentioned discrepancy that a distributive complemented poset need not be orthomodular, we define a concept that is a bit more general and which was introduced in \cite{CL22} (cf.\ also the paper \cite{CFL}):

\begin{definition}\label{def1}
A {\em generalized orthomodular poset} is a bounded poset $\mathbf P=(P,\leq,{}',0,1)$ with an antitone involution that is a complementation that satisfies the condition
\begin{enumerate}
\item[{\rm(GOM)}] $x\leq y$ implies $U(y)=U\big(x,L(y,x')\big)$.
\end{enumerate}
\end{definition}

It is worth noticing that (GOM) is equivalent to its dual
\[
x\leq y\text{ implies }L(x)=L\big(y,U(x,y')\big),
\]
and if the poset is {\em orthogonal}, i.e.\ if for all $x,y\in P$ with $x\leq y'$ there exists $x\vee y$, then (GOM) is equivalent to (OM) and hence $\mathbf P$ is an orthomodular poset.

Recall that a {\em Boolean poset} is a distributive complemented poset.

By Lemma~\ref{lem4} the complementation in a Boolean poset is unique and antitone. This fact will be used in proofs of our results in Section~4. It is easy to prove the following assertion.

\begin{proposition}\label{prop2}
Let $\mathbf B=(B,\leq,{}',0,1)$ be a Boolean poset. Then $\mathbf B$ is a generalized orthomodular poset.
\end{proposition}

\begin{proof}
Since $x$ and $x''$ are complements of $x'$, we obtain $x''\approx x$ by Lemma~\ref{lem4} (i). According to Lemma~\ref{lem4} (ii), $'$ is antitone. Finally, if $x\leq y$ then
\[
U(y)=ULU(y)=UL\big(U(y),1\big)=UL\big(U(x,y),U(x,x')\big)=U\big(x,L(y,x')\big)
\]
using distributivity of $\mathbf B$.
\end{proof}

\begin{example}
The poset depicted in Fig.~1 is a non-lattice Boolean poset and hence a generalized orthomodular poset according to Proposition~\ref{prop2}.

\vspace*{-4mm}

\begin{center}
\setlength{\unitlength}{7mm}
\begin{picture}(8,10)
\put(4,1){\circle*{.3}}
\put(1,3){\circle*{.3}}
\put(3,3){\circle*{.3}}
\put(5,3){\circle*{.3}}
\put(7,3){\circle*{.3}}
\put(1,5){\circle*{.3}}
\put(7,5){\circle*{.3}}
\put(1,7){\circle*{.3}}
\put(3,7){\circle*{.3}}
\put(5,7){\circle*{.3}}
\put(7,7){\circle*{.3}}
\put(4,9){\circle*{.3}}
\put(4,1){\line(-3,2)3}
\put(4,1){\line(-1,2)1}
\put(4,1){\line(1,2)1}
\put(4,1){\line(3,2)3}
\put(4,9){\line(-3,-2)3}
\put(4,9){\line(-1,-2)1}
\put(4,9){\line(1,-2)1}
\put(4,9){\line(3,-2)3}
\put(1,3){\line(0,1)4}
\put(1,3){\line(1,1)4}
\put(3,3){\line(-1,1)2}
\put(3,3){\line(1,1)4}
\put(5,3){\line(-1,1)4}
\put(5,3){\line(1,1)2}
\put(7,3){\line(-1,1)4}
\put(7,3){\line(0,1)4}
\put(1,5){\line(1,1)2}
\put(7,5){\line(-1,1)2}
\put(3.85,.3){$0$}
\put(.35,2.85){$a$}
\put(2.35,2.85){$b$}
\put(5.4,2.85){$c$}
\put(7.4,2.85){$d$}
\put(.35,4.85){$e$}
\put(7.4,4.85){$e'$}
\put(.35,6.85){$d'$}
\put(2.35,6.85){$c'$}
\put(5.4,6.85){$b'$}
\put(7.4,6.85){$a'$}
\put(3.85,9.4){$1$}
\put(3.2,-.75){{\rm Fig.~1}}
\end{picture}
\end{center}

\vspace*{4mm}

This poset is not an orthomodular poset since $a\leq c'$, but $a\vee c$ does not exist.
\end{example}

\section{The smallest non-lattice orthomodular poset}

As mentioned in the introduction, as far as we know, the smallest non-lattice orthomodular poset is not known up to now. Sometimes the following 20-element non-lattice orthomodular poset was considered (Fig.~2). It is the poset of all subsets $A$ of the set $\{1,\ldots,6\}$ having an even number of elements and satisfying $|A\cap\{1,2,3\}|=|A\cap\{4,5,6\}|$.

\vspace*{-2mm}

\begin{center}
\setlength{\unitlength}{7mm}
\begin{picture}(18,8)
\put(9,1){\circle*{.3}}
\put(1,3){\circle*{.3}}
\put(3,3){\circle*{.3}}
\put(5,3){\circle*{.3}}
\put(7,3){\circle*{.3}}
\put(9,3){\circle*{.3}}
\put(11,3){\circle*{.3}}
\put(13,3){\circle*{.3}}
\put(15,3){\circle*{.3}}
\put(17,3){\circle*{.3}}
\put(1,5){\circle*{.3}}
\put(3,5){\circle*{.3}}
\put(5,5){\circle*{.3}}
\put(7,5){\circle*{.3}}
\put(9,5){\circle*{.3}}
\put(11,5){\circle*{.3}}
\put(13,5){\circle*{.3}}
\put(15,5){\circle*{.3}}
\put(17,5){\circle*{.3}}
\put(9,7){\circle*{.3}}
\put(1,3){\line(0,2)2}
\put(1,3){\line(1,1)2}
\put(1,3){\line(3,1)6}
\put(1,3){\line(4,1)8}
\put(3,3){\line(-1,1)2}
\put(3,3){\line(1,1)2}
\put(3,3){\line(2,1)4}
\put(3,3){\line(4,1)8}
\put(5,3){\line(-1,1)2}
\put(5,3){\line(0,1)2}
\put(5,3){\line(2,1)4}
\put(5,3){\line(3,1)6}
\put(7,3){\line(-3,1)6}
\put(7,3){\line(-2,1)4}
\put(7,3){\line(3,1)6}
\put(7,3){\line(4,1)8}
\put(9,3){\line(-4,1)8}
\put(9,3){\line(-2,1)4}
\put(9,3){\line(2,1)4}
\put(9,3){\line(4,1)8}
\put(11,3){\line(-4,1)8}
\put(11,3){\line(-3,1)6}
\put(11,3){\line(2,1)4}
\put(11,3){\line(3,1)6}
\put(13,3){\line(-3,1)6}
\put(13,3){\line(-2,1)4}
\put(13,3){\line(0,1)2}
\put(13,3){\line(1,1)2}
\put(15,3){\line(-4,1)8}
\put(15,3){\line(-2,1)4}
\put(15,3){\line(-1,1)2}
\put(15,3){\line(1,1)2}
\put(17,3){\line(-4,1)8}
\put(17,3){\line(-3,1)6}
\put(17,3){\line(-1,1)2}
\put(17,3){\line(0,1)2}
\put(9,1){\line(-4,1)8}
\put(9,1){\line(-3,1)6}
\put(9,1){\line(-2,1)4}
\put(9,1){\line(-1,1)2}
\put(9,1){\line(0,1)2}
\put(9,1){\line(1,1)2}
\put(9,1){\line(2,1)4}
\put(9,1){\line(3,1)6}
\put(9,1){\line(4,1)8}
\put(9,7){\line(-4,-1)8}
\put(9,7){\line(-3,-1)6}
\put(9,7){\line(-2,-1)4}
\put(9,7){\line(-1,-1)2}
\put(9,7){\line(0,-1)2}
\put(9,7){\line(1,-1)2}
\put(9,7){\line(2,-1)4}
\put(9,7){\line(3,-1)6}
\put(9,7){\line(4,-1)8}
\put(8.85,.25){$\emptyset$}
\put(.85,2.3){$a$}
\put(2.85,2.3){$b$}
\put(4.85,2.3){$c$}
\put(6.85,2.3){$d$}
\put(8.85,2.3){$e$}
\put(10.85,2.3){$f$}
\put(12.85,2.3){$g$}
\put(14.85,2.3){$h$}
\put(16.85,2.3){$i$}
\put(.8,5.4){$i'$}
\put(2.8,5.4){$h'$}
\put(4.8,5.4){$g'$}
\put(6.8,5.4){$f'$}
\put(8.8,5.4){$e'$}
\put(10.8,5.4){$d'$}
\put(12.8,5.4){$c'$}
\put(14.8,5.4){$b'$}
\put(16.8,5.4){$a'$}
\put(8.75,7.4){$N$}
\put(8.2,-.75){{\rm Fig.~2}}
\end{picture}
\end{center}

\vspace*{4mm}

Here $a=\{1,4\}$, $b=\{1,5\}$, $c=\{1,6\}$, $d=\{2,4\}$, $e=\{2,5\}$, $f=\{2,6\}$, $g=\{3,4\}$, $h=\{3,5\}$, $i=\{3,6\}$, $a'=\{2,3,5,6\}$, $b'=\{2,3,4,6\}$, $c'=\{2,3,4,5\}$, $d'=\{1,3,5,6\}$, $e'=\{1,3,4,6\}$, $f'=\{1,3,4,5\}$, $g'=\{1,2,5,6\}$, $h'=\{1,2,4,6\}$, $i'=\{1,2,4,5\}$ and $N=\{1,\ldots,6\}$.
For
\[
P:=\{A\subseteq N\mid|A\cap\{1,2,3\}|=|A\cap\{4,5,6\}|\},
\]
the poset $\mathbf P=(P,\subseteq,{}',\emptyset,N)$ is not a lattice since e.g.\ $a\vee b$ does not exist. Note that $\mathbf P$ is the smallest orthomodular subposet of the orthomodular poset $(Q,\subseteq,{}',\emptyset,N)$ with $Q:=\{A\in 2^N\mid A\text{ has an even number of elements}\}$ containing $a=\{1,4\}$ and $b=\{1,5\}$.

However, we will prove the following result.

\begin{theorem}
The smallest non-lattice orthomodular poset is depicted in Fig.~3 and is unique up to isomorphism.

\vspace*{-4mm}

\begin{center}
\setlength{\unitlength}{7mm}
\begin{picture}(16,8)
\put(8,1){\circle*{.3}}
\put(1,3){\circle*{.3}}
\put(3,3){\circle*{.3}}
\put(5,3){\circle*{.3}}
\put(7,3){\circle*{.3}}
\put(9,3){\circle*{.3}}
\put(11,3){\circle*{.3}}
\put(13,3){\circle*{.3}}
\put(15,3){\circle*{.3}}
\put(1,5){\circle*{.3}}
\put(3,5){\circle*{.3}}
\put(5,5){\circle*{.3}}
\put(7,5){\circle*{.3}}
\put(9,5){\circle*{.3}}
\put(11,5){\circle*{.3}}
\put(13,5){\circle*{.3}}
\put(15,5){\circle*{.3}}
\put(8,7){\circle*{.3}}
\put(1,3){\line(0,2)2}
\put(1,3){\line(1,1)2}
\put(1,3){\line(2,1)4}
\put(1,3){\line(4,1)8}
\put(3,3){\line(-1,1)2}
\put(3,3){\line(0,1)2}
\put(3,3){\line(2,1)4}
\put(3,3){\line(4,1)8}
\put(5,3){\line(-2,1)4}
\put(5,3){\line(4,1)8}
\put(7,3){\line(-3,1)6}
\put(7,3){\line(4,1)8}
\put(9,3){\line(-3,1)6}
\put(9,3){\line(2,1)4}
\put(11,3){\line(-4,1)8}
\put(11,3){\line(2,1)4}
\put(13,3){\line(-4,1)8}
\put(13,3){\line(-3,1)6}
\put(13,3){\line(0,1)2}
\put(13,3){\line(1,1)2}
\put(15,3){\line(-3,1)6}
\put(15,3){\line(-2,1)4}
\put(15,3){\line(-1,1)2}
\put(15,3){\line(0,1)2}
\put(8,1){\line(-7,2)7}
\put(8,1){\line(-5,2)5}
\put(8,1){\line(-3,2)3}
\put(8,1){\line(-1,2)1}
\put(8,1){\line(1,2)1}
\put(8,1){\line(3,2)3}
\put(8,1){\line(5,2)5}
\put(8,1){\line(7,2)7}
\put(8,7){\line(-7,-2)7}
\put(8,7){\line(-5,-2)5}
\put(8,7){\line(-3,-2)3}
\put(8,7){\line(-1,-2)1}
\put(8,7){\line(1,-2)1}
\put(8,7){\line(3,-2)3}
\put(8,7){\line(5,-2)5}
\put(8,7){\line(7,-2)7}
\put(7.85,.25){$0$}
\put(.85,2.3){$a$}
\put(2.85,2.3){$b$}
\put(4.85,2.3){$c$}
\put(6.85,2.3){$d$}
\put(8.85,2.3){$e$}
\put(10.85,2.3){$f$}
\put(12.85,2.3){$g$}
\put(14.85,2.3){$h$}
\put(.8,5.4){$h'$}
\put(2.8,5.4){$g'$}
\put(4.8,5.4){$f'$}
\put(6.8,5.4){$e'$}
\put(8.8,5.4){$d'$}
\put(10.8,5.4){$c'$}
\put(12.8,5.4){$b'$}
\put(14.8,5.4){$a'$}
\put(7.85,7.4){$1$}
\put(7.2,-.75){{\rm Fig.~3}}
\end{picture}
\end{center}

\vspace*{1mm}

\end{theorem}

\begin{proof}
Let $\mathbf P=(P,\leq,{}',0,1)$ be a minimal non-lattice orthomodular poset. Then there exist $a,b\in P$ having no supremum. Let $g'$ and $h'$ be two minimal upper bounds of $a$ and $b$. If $g'$ and $h'$ would have an infimum then they would not be minimal upper bounds of $a$ and $b$. Hence $g'$ and $h'$ have no infimum. Thus the Hasse diagram of $\mathbf P$ must contain the configuration shown in Fig.~4:

\vspace*{-4mm}

\begin{center}
\setlength{\unitlength}{7mm}
\begin{picture}(2,4)
\put(0,1){\circle*{.3}}
\put(0,3){\circle*{.3}}
\put(2,1){\circle*{.3}}
\put(2,3){\circle*{.3}}
\put(0,1){\line(0,1)2}
\put(0,1){\line(1,1)2}
\put(2,1){\line(-1,1)2}
\put(2,1){\line(0,1)2}
\put(-.15,.3){$a$}
\put(1.85,.3){$b$}
\put(-.2,3.4){$h'$}
\put(1.8,3.4){$g'$}
\put(.2,-.75){{\rm Fig.~4}}
\end{picture}
\end{center}

\vspace*{4mm}

Since $\mathbf P$ is bounded and its unary operation $'$ is an antitone involution that is a complementation, $\mathbf P$ must contain the configuration visualized in Fig.~5 and $P$ must have an even number of elements. We also conclude that $a'$ and $b'$ have no infimum and $g$ and $h$ have no supremum.

\vspace*{-4mm}

\begin{center}
\setlength{\unitlength}{7mm}
\begin{picture}(6,8)
\put(0,3){\circle*{.3}}
\put(0,5){\circle*{.3}}
\put(2,3){\circle*{.3}}
\put(2,5){\circle*{.3}}
\put(3,1){\circle*{.3}}
\put(4,3){\circle*{.3}}
\put(4,5){\circle*{.3}}
\put(3,7){\circle*{.3}}
\put(6,3){\circle*{.3}}
\put(6,5){\circle*{.3}}
\put(3,1){\line(-3,2)3}
\put(3,1){\line(-1,2)1}
\put(3,1){\line(1,2)1}
\put(3,1){\line(3,2)3}
\put(3,7){\line(-3,-2)3}
\put(3,7){\line(-1,-2)1}
\put(3,7){\line(1,-2)1}
\put(3,7){\line(3,-2)3}
\put(0,3){\line(0,1)2}
\put(0,3){\line(1,1)2}
\put(2,3){\line(-1,1)2}
\put(2,3){\line(0,1)2}
\put(4,3){\line(0,1)2}
\put(4,3){\line(1,1)2}
\put(6,3){\line(-1,1)2}
\put(6,3){\line(0,1)2}
\put(2.85,.25){$0$}
\put(-.15,2.3){$a$}
\put(1.85,2.3){$b$}
\put(3.85,2.3){$g$}
\put(5.85,2.3){$h$}
\put(-.2,5.4){$h'$}
\put(1.8,5.4){$g'$}
\put(3.8,5.4){$b'$}
\put(5.8,5.4){$a'$}
\put(2.85,7.4){$1$}
\put(2.2,-.75){{\rm Fig.~5}}
\end{picture}
\end{center}

\vspace*{4mm}

It is clear that these 10 elements are pairwise distinct. Let us mention that this poset is orthogonal. Put
\begin{align*}
c & :=h'\wedge b', \\
d & :=h'\wedge a', \\
e & :=g'\wedge b', \\
f & :=g'\wedge a'.
\end{align*}
Because of the orthomodularity we have
\begin{align*}
h' & =b\vee c, \\
h' & =a\vee d, \\
g' & =b\vee e, \\
g' & =a\vee f.
\end{align*}
Using the facts $b\neq0$, $h\neq0$, $h\neq b'$ and that neither $a\vee b$ nor $g\vee h$ exists, we can prove $c\neq0,a,b,g,h,a',b',g',h',1$.
\begin{quote}
$c=0$ would imply $h'=b\vee0=b$, a contradiction, \\
$c=a$ would imply $a=h'\wedge b'\leq b'$ and hence $a\vee b$ would exist, a contradiction, \\
$c=b$ would imply $h'=b\vee b=b$, a contradiction, \\
$c=g$ would imply $g=h'\wedge b'\leq h'$ and hence $g\vee h$ would exist, a contradiction, \\
$c=h$ would imply $h\leq b\vee h=h'$ and hence $h=h\wedge h'=0$, a contradiction, \\
$c=a'$ would imply $a'=h'\wedge b'\leq b'$ and hence $a\vee b$ would exist, a contradiction, \\
$c=b'$ would imply $h'=b\vee b'=1$, a contradiction, \\
$c=g'$ would imply $g'\leq b\vee g'=h'$ and hence $g\vee h$ would exist, a contradiction, \\
$c=h'$ would imply $b\leq b\vee h'=h'=h'\wedge b'\leq b'$ and hence $b=b\wedge b'=0$, a~contradiction, \\
$c=1$ would imply $h'=b\vee1=1$, a contradiction.
\end{quote}
This shows $c\neq0,a,b,g,h,a',b',g',h',1$. Hence also $c'$ is different from these 10 elements. Because of symmetry reasons also $d,e,f,d',e',f'$ are different from these 10 elements. Altogether, we have shown that any of the elements $c,d,e,f,c',d',e',f'$ is different from $0,a,b,g,h,a',b',g',h',1$. Using the facts $c\neq0$, $h\neq0$ and that $g\vee h$ does not exist, we can prove $c\neq c',d',e',f'$.
\begin{quote}
$c=c'$ would imply $c=c\wedge c'=0$, a contradiction, \\
$c=d'$ would imply $h\leq h\vee a=d'=h'\wedge b\leq h'$ and hence $h=h\wedge h'=0$, a~contradiction, \\
$c=e'$ would imply $g\leq g\vee b=e'=h'\wedge b'\leq h'$ and hence $g\vee h$ would exist, a contradiction, \\
$c=f'$ would imply $g\leq g\vee a=f'=h'\wedge b'\leq h'$ and hence $g\vee h$ would exist, a contradiction.
\end{quote}
This shows $c\neq c',d',e',f'$. Because of symmetry reasons also $d,e,f$ are different from $c',d',e',f'$. Altogether, we have shown that any of the elements $c,d,e,f$ is different from $c',d',e',f'$. Using the facts $a\neq b$ and $g\neq h$, we can prove that $c,d,e,f$ are pairwise different.
\begin{quote}
$c=d$ would imply $h'\wedge b'=h'\wedge a'$ and hence $a'=h\vee(a'\wedge h')=h\vee(b'\wedge h')=b'$, a contradiction. \\
$c=e$ would imply $g'=b\vee e=b\vee c=h'$, a contradiction. \\
$c=f$ would imply $a\leq h'$, $f=c=h'\wedge b'\leq h'$, $b\leq g'$ and $c=f=g'\wedge a'\leq g'$ and hence $g'=a\vee f\leq h'=b\vee c\leq g'$ whence $g'=h'$, a contradiction. \\
$d=e$ would imply $b\leq h'$, $e=d=h'\wedge a'\leq h'$, $a\leq g'$ and $d=e=g'\wedge b'\leq g'$ and hence $g'=b\vee e\leq h'=a\vee d\leq g'$ whence $g'=h'$, a contradiction. \\
$d=f$ would imply $g'=a\vee f=a\vee d=h'$, a contradiction. \\
$e=f$ would imply $g'\wedge b'=g'\wedge a'$ and hence $a'=g\vee(a'\wedge g')=g\vee(b'\wedge g')=b'$, a contradiction.
\end{quote}
This shows that $c,d,e,f$ are pairwise different. Hence also $c',d',e',f'$ are pairwise different. Altogether, we have proved that the 18 elements
\[
0,a,b,c,d,e,f,g,h,a',b',c',d',e',f',g',h',1
\]
are pairwise different. This shows that $\mathbf P$ must contain the poset depicted in Fig.~3.
But this is already a non-lattice orthomodular poset and hence the smallest one with respect to the number of its elements. We need to show that it is unique up to isomorphism. If there would exist another 18-element orthomodular poset not isomorphic to $\mathbf P$ then its Hasse diagram must contain an edge which is not included in the Hasse diagram of $\mathbf P$. We can check this case. Consider that the orthomodular poset in question contains e.g.\ the edge $(a,b')$ and hence also $(b,a')$. Then (OM) is violated since $h'=a\vee(h'\wedge a')$ does not hold. If it would contain e.g.\ the edges $(c,g')$ and $(g,c')$ then similarly $g'=c\vee(g'\wedge c')$ does not hold. If it would contain e.g.\ $(d,b)$ and $(b',d')$ then (OM) would also be violated since
\begin{quote}
by adding the edge $(d,b)$ such that $d\leq b$ we would get that $h'=a\vee(h'\wedge a')$ would not hold,\\
by adding the edge $(d,b)$ such that $b\leq d$ we would get that $g'=a\vee(g'\wedge a')$ would not hold.
\end{quote}
All the remaining cases can be checked analogously as the previous ones. In all the cases we would reach a contradiction which proves that the poset $\mathbf P$ is unique up to isomorphism.
\end{proof}

\section{Horizontal sums}

The aim of this section is to describe a construction of generalized orthomodular posets by means of so-called horizontal sums. For the reader's convenience, let us recall this concept.

Let $\mathbf P_i=(P_i,\leq,{}',0,1),i\in I,$ be a non-empty family of bounded posets with an antitone involution. By the {\em horizontal sum} of the $\mathbf P_i$ we mean a poset $\mathbf P$ that is the union of disjoint copies of the posets $\mathbf P_i$ where the bottom and top elements of the $\mathbf P_i$, respectively, are identified.

\begin{proposition}\label{prop1}
Let $\mathbf P_i=(P_i,\leq,{}',0,1),i\in I,$ be a non-empty family of generalized orthomodular posets. Then the horizontal sum $(P,\leq,{}',0,1)$ of the $\mathbf P_i,i\in I,$ is a generalized orthomodular poset.
\end{proposition}

\begin{proof}
If $a,b\in P$ and $a\leq b$ then there exists some $i\in I$ with $a,b\in P_i$, thus (GOM) surely holds. If there does not exist some $i\in I$ with $a,b\in P_i$ then $a\parallel b$ and hence (GOM) is satisfied trivially for these elements $a$ and $b$.
\end{proof}

Let us note that if $|I|>1$ and each $\mathbf P_i$ is non-trivial (i.e.\ has more than two elements) then the horizontal sum of the $\mathbf P_i$ is a non-distributive generalized orthomodular poset. Namely, if $j,k\in I$, $j\neq k$, $a\in P_j\setminus\{0,1\}$ and $b\in P_k\setminus\{0,1\}$ then
\[
L\big(U(a,a'),b\big)=L(1,b)=L(b)\neq L(0)=LU(0)=LU(0,0)=LU\big(L(a,b),L(a',b)\big).
\]

\begin{corollary}
The horizontal sum of a family of Boolean posets is a generalized orthomodular poset.
\end{corollary}

\begin{proof}
By Proposition~\ref{prop2}, every Boolean poset is a generalized orthomodular poset. The rest follows from Proposition~\ref{prop1}.
\end{proof}

\begin{example}
The orthomodular poset $\mathbf P$ depicted in Fig.~3 is not a horizontal sum of Boolean posets.
\end{example}

It is a question if every generalized orthomodular poset is the horizontal sum of Boolean posets. In the next example we show that this is not the case.

\begin{example}
Consider the poset $\mathbf P$ depicted in Fig.~6:

\vspace*{-4mm}

\begin{center}
\setlength{\unitlength}{7mm}
\begin{picture}(8,8)
\put(4,1){\circle*{.3}}
\put(1,3){\circle*{.3}}
\put(3,3){\circle*{.3}}
\put(5,3){\circle*{.3}}
\put(7,3){\circle*{.3}}
\put(1,5){\circle*{.3}}
\put(3,5){\circle*{.3}}
\put(5,5){\circle*{.3}}
\put(7,5){\circle*{.3}}
\put(4,7){\circle*{.3}}
\put(4,1){\line(-3,2)3}
\put(4,1){\line(-1,2)1}
\put(4,1){\line(1,2)1}
\put(4,1){\line(3,2)3}
\put(4,7){\line(-3,-2)3}
\put(4,7){\line(-1,-2)1}
\put(4,7){\line(1,-2)1}
\put(4,7){\line(3,-2)3}
\put(1,3){\line(0,1)2}
\put(1,3){\line(1,1)2}
\put(1,3){\line(2,1)4}
\put(3,3){\line(-1,1)2}
\put(3,3){\line(0,1)2}
\put(3,3){\line(2,1)4}
\put(5,3){\line(-2,1)4}
\put(5,3){\line(0,1)2}
\put(5,3){\line(1,1)2}
\put(7,3){\line(-2,1)4}
\put(7,3){\line(-1,1)2}
\put(7,3){\line(0,1)2}
\put(3.85,.3){$0$}
\put(.35,2.85){$a$}
\put(2.35,2.85){$b$}
\put(5.4,2.85){$c$}
\put(7.4,2.85){$d$}
\put(.35,4.85){$d'$}
\put(2.35,4.85){$c'$}
\put(5.4,4.85){$b'$}
\put(7.4,4.85){$a'$}
\put(3.85,7.4){$1$}
\put(3.2,-.75){{\rm Fig.~6}}
\end{picture}
\end{center}

\vspace*{4mm}

$\mathbf P$ is a non-lattice Boolean poset that is not orthomodular since $a\leq d'$, but $a\vee(d'\wedge a')$ is not defined. If we consider the horizontal sum of $\mathbf P$ and some non-distributive generalized orthomodular poset {\rm(}e.g.\ the poset visualized in Fig.~3{\rm)} then we obtain a non-lattice generalized orthomodular poset being not distributive, not a horizontal sum of Boolean posets and not orthomodular.
\end{example}

In what follows we will study generalized orthomodular posets that are horizontal sums of Boolean ones. For this purpose we introduce the compatibility relation analogously as it was done in orthomodular lattices, see e.g.\ \cite K.

In orthomodular lattices $(L,\vee,\wedge)$ the compatibility relation ${\rm C}$ is defined as follows:
\[
a\mathrel{{\rm C}}b\text{ if }a=(a\wedge b)\vee(a\wedge b')
\]
($a,b\in L$). Of course, in a Boolean algebra every two elements are compatible. For our reasons, define the relation ${\rm C}$ in a generalized orthomodular poset $(P,\leq,{}',0,1)$ as follows:
\[
a\mathrel{{\rm C}}b\text{ if }U(a)=U\big(L(a,b),L(a,b')\big)
\]
($a,b\in P$). Then $a,b$ are called {\em compatible} and ${\rm C}$ is called the {\em compatibility relation}.

\begin{lemma}\label{lem1}
Let $(P,\leq,{}',0,1)$ be a generalized orthomodular poset and $a,b\in P$. Then the following hold:
\begin{enumerate}[{\rm(i)}]
\item $a\mathrel{{\rm C}}b$ if and only if $a\mathrel{{\rm C}}b'$,
\item $a\leq b$ implies $a\mathrel{{\rm C}}b$,
\item if $\{a,b\}\cap\{0,1\}\neq\emptyset$ then $a\mathrel{{\rm C}}b$.
\end{enumerate}
\end{lemma}

\begin{proof}
\
\begin{enumerate}[(i)]
\item This is clear.
\item $a\leq b$ implies $U(a)=UL(a)=U\big(L(a),L(a,b')\big)=U\big(L(a,b),L(a,b')\big)$.
\item If $a=0$ or $b=1$ then $a\mathrel{{\rm C}}b$ follows from (ii), if $a=1$ then $a\mathrel{{\rm C}}b$ follows from
\begin{align*}
U(a) & =U(1)=\{1\}=U(b,b')=U\big(L(b),L(b')\big)=U\big(L(1,b),L(1,b')\big)= \\
     & =U\big(L(a,b),L(a,b')\big)
\end{align*}
and if $b=0$ then $a\mathrel{{\rm C}}b$ follows from (i) and (ii).
\end{enumerate}
\end{proof}

The following result is almost evident.

\begin{lemma}\label{lem2}
Let $(B,\leq,{}',0,1)$ be a Boolean poset and $a,b\in B$. Then $a\mathrel{{\rm C}}b$.
\end{lemma}

\begin{proof}
We have $U(a)=UL(a)=UL(a,1)=UL\big(a,U(b,b')\big)=U\big(L(a,b),L(a,b')\big)$.
\end{proof}

However, we can prove a more interesting and important result.

\begin{theorem}\label{th1}
Let $(P,\leq,{}',0,1)$ be the horizontal sum of the Boolean posets $(B_i,\leq,{}',0,$ $1),i\in I$, and $a,b\in P$. Then the following are equivalent:
\begin{enumerate}[{\rm(i)}]
\item $a\mathrel{{\rm C}}b$,
\item there exists some $i\in I$ with $a,b\in B_i$.
\end{enumerate}
\end{theorem}

\begin{proof}
$\text{}$ \\
(i) $\Rightarrow$ (ii): \\
If there would exist no $i\in I$ with $a,b\in B_i$ then there would exist $j,k\in I$ with $j\neq k$, $a\in B_j\setminus\{0,1\}$ and $b\in B_k\setminus\{0,1\}$ which would imply
\[
U(a)\neq U(0)=U(0,0)=U\big(L(a,b),L(a,b')\big),
\]
a contradiction. \\
(ii) $\Rightarrow$ (i): \\
If $\{a,b\}\cap\{0,1\}\neq\emptyset$ then $a\mathrel{{\rm C}}b$ according to Lemma~\ref{lem1}. Now assume $a,b\in B_i\setminus\{0,1\}$. Because of Lemma~\ref{lem2} we have
\[
U\big(L(a,b),L(a,b')\big)\cap B_i=U\big(L(a,b)\cap B_i,L(a,b')\cap B_i\big)\cap B_i=U(a)\cap B_i=U(a)
\]
which implies $L(a,b)\cup L(a,b')\neq\{0\}$ and hence
\[
U(a)=U\big(L(a,b),L(a,b')\big)\cap B_i=U\big(L(a,b),L(a,b')\big),
\]
i.e.\ $a\mathrel{{\rm C}}b$.
\end{proof}

For orthomodular lattices $(L,\vee,\wedge)$ the commutator $c(x,y)$ was introduced as follows:
\[
c(x,y):=(x\wedge y)\vee(x\wedge y')\vee(x'\wedge y)\vee(x'\wedge y')
\]
for all $x,y\in L$ (cf.\ e.g.\ \cite K). In generalized orthomodular poset $(P,\leq,{}',0,1)$ we define analogously
\[
c(x,y):=\Min U\big(L(x,y),L(x,y'),L(x',y),L(x',y')\big)
\]
for all $x,y\in P$. Here and in the following $\Min A$ for a subset $A$ of a poset means the set of all minimal elements of $A$, and $\Min U(A)$ means $\Min\big(U(A)\big)$. Let us note that it may happen that $\Min A=\emptyset$ if $A$ has no minimal elements.

In the following we often identify singletons with their unique element.

\begin{lemma}\label{lem3}
Let $(P,\leq,{}',0,1)$ be a generalized orthomodular poset and $a,b\in P$. Then the following hold:
\begin{enumerate}[{\rm(i)}]
\item $c(a,b)=c(b,a)$,
\item $c(a,b)=c(a,b')=c(a',b)=c(a',b')$,
\item $c(0,b)=c(1,b)=1$,
\item if $a\mathrel{{\rm C}}b$ and $a'\mathrel{{\rm C}}b$ then $c(a,b)=1$.
\end{enumerate}
\end{lemma}

\begin{proof}
\
\begin{enumerate}
\item[(i)] and (ii) are clear.
\item[(iii)] According to (ii) we have
\[
c(0,b)=c(1,b)=\Min U\big(L(1,b),L(1,b'),L(0,b),L(0,b')\big)=\Min U(b,b',0)=1.
\]
\item[(iv)] If $a\mathrel{{\rm C}}b$ and $a'\mathrel{{\rm C}}b$ then
\begin{align*}
c(a,b) & =\Min U\big(L(a,b),L(a,b'),L(a',b),L(a',b')\big)= \\
       & =\Min\Big(U\big(L(a,b),L(a,b')\big)\cap U\big(L(a',b),L(a',b')\big)\Big)=\Min\big(U(a)\cap U(a')\big)= \\
			 & =\Min U(a,a')=1.
\end{align*}
\end{enumerate}
\end{proof}

\begin{corollary}\label{cor1}
Let $(B,\leq,{}',0,1)$ be a Boolean poset and $a,b\in B$. Then $c(a,b)=1$.
\end{corollary}

\begin{proof}
We have $a,a',b\in B$ and hence $a\mathrel{{\rm C}}b$ and $a'\mathrel{{\rm C}}b$ according to Lemma~\ref{lem2} which implies $c(a,b)=1$ by Lemma~\ref{lem3}.
\end{proof}

Now we prove a result similar to Theorem~\ref{th1} for the commutator instead of compatibility.

\begin{theorem}\label{th2}
Let $(P,\leq,{}',0,1)$ be the horizontal sum of the Boolean posets $(B_i,\leq,{}',0,1),$ $i\in I$, and $a,b\in P$. Then
\[
c(a,b)=\left\{
\begin{array}{ll}
1 & \text{if there exists some }i\in I\text{ with }a,b\in B_i \\
0 & \text{otherwise}.
\end{array}
\right.
\]
Hence $c(a,b)=1$ if and only if there exists some $i\in I$ with $a,b\in B_i$.
\end{theorem}

\begin{proof}
First assume there exists some $i\in I$ with $a,b\in B_i$. If $\{a,b\}\cap\{0,1\}\neq\emptyset$ then $c(a,b)=1$ according to Lemma~\ref{lem3}. Now assume $a,b\in B_i\setminus\{0,1\}$. Because of Corollary~\ref{cor1} we have
\begin{align*}
& U\big(L(a,b)\cup L(a,b')\cup L(a',b)\cup L(a',b')\big)\cap B_i= \\
& =U\big(L(a,b)\cap B_i,L(a,b')\cap B_i,L(a',b)\cap B_i,L(a',b')\cap B_i\big)\cap B_i=1
\end{align*}
which implies $L(a,b)\cup L(a,b')\cup L(a',b)\cup L(a',b')\neq0$ and hence
\begin{align*}
c(a,b) & =\Min U\big(L(a,b),L(a,b'),L(a',b),L(a',b')\big)= \\
       & =\Min\Big(U\big(L(a,b)\cup L(a,b')\cup L(a',b)\cup L(a',b')\big)\cap B_i\Big)=1.
\end{align*}
Conversely, assume there exists no $i\in I$ with $a,b\in B_i$. Then there exist $j,k\in I$ with $j\neq k$, $a\in B_j\setminus\{0,1\}$ and $b\in B_k\setminus\{0,1\}$ and hence
\[
c(a,b)=\Min U\big(L(a,b),L(a,b'),L(a',b),L(a',b')\big)=\Min U(0)=0.
\]
\end{proof}

It is worth noticing that the assumptions of Theorem~\ref{th2} are essential. Namely if the generalized orthomodular poset $(P,\leq,{}',0,1)$ is neither Boolean nor a horizontal sum of such posets then for $x,y\in P$ it may happen that $c(x,y)$ differs from both $0$ and $1$, see the following example.

\begin{example}
\
\begin{enumerate}[{\rm(i)}]
\item Consider the orthomodular poset $(P,\leq,{}',0,1)$ depicted in Fig.~3. Then we compute
\begin{align*}
c(a,b) & =\Min U\big(L(a,b),L(a,b'),L(a',b),L(a',b')\big)=\Min U(0,0,0,0,g,h\})= \\
       & =\Min U(g,h)=\{a',b'\}
\end{align*}
which differs from both $0$ and $1$.
\item However, the condition from Theorem~\ref{th2} does not characterize the class of generalized orthomodular posets that are horizontal sums of Boolean posets. For example, consider the ortholattice $\mathbf O_6=(O_6,\leq,{}',0,1)$ visualized in Fig.~7:

\vspace*{-3mm}

\begin{center}
\setlength{\unitlength}{7mm}
\begin{picture}(4,8)
\put(2,1){\circle*{.3}}
\put(1,3){\circle*{.3}}
\put(3,3){\circle*{.3}}
\put(1,5){\circle*{.3}}
\put(3,5){\circle*{.3}}
\put(2,7){\circle*{.3}}
\put(2,1){\line(-1,2)1}
\put(2,1){\line(1,2)1}
\put(1,5){\line(0,-1)2}
\put(1,5){\line(1,2)1}
\put(3,5){\line(0,-1)2}
\put(3,5){\line(-1,2)1}
\put(1.85,.3){$0$}
\put(.35,2.85){$a$}
\put(3.4,2.85){$b$}
\put(.35,4.85){$b'$}
\put(3.4,4.85){$a'$}
\put(1.85,7.4){$1$}
\put(1.2,-.75){{\rm Fig.~7}}
\end{picture}
\end{center}

\vspace*{4mm}

One can easily check that $c(x,y)=1$ for all $x,y\in O_6$ {\rm(}if we define the commutator in ortholattices in the same way as it was done for orthomodular lattices{\rm)}. Of course, this lattice is not a horizontal sum of Boolean posets, but it is also not a generalized orthomodular poset.
\end{enumerate}
\end{example}

On the other hand, for arbitrary generalized orthomodular posets we can prove the following result.

\begin{proposition}
Let $(P,\leq,{}',0,1)$ be a generalized orthomodular poset. Then the following are equivalent:
\begin{enumerate}[{\rm(i)}]
\item $c(x,y)\in\{0,1\}$ for all $x,y\in P$,
\item If $x,y\in P$ then either $L(x,y)=L(x,y')=L(x',y)=L(x',y')=0$ or \\
$U\big(L(x,y),L(x,y'),L(x',y),L(x',y')\big)=1$.
\end{enumerate}
\end{proposition}

\begin{proof}
Let $a,b\in P$. Then the following are equivalent:
\begin{align*}
                                         c(a,b) & =0, \\
\Min U\big(L(a,b),L(a,b'),L(a',b),L(a',b')\big) & =0, \\
                                              0 & \in U\big(L(a,b),L(a,b'),L(a',b),L(a',b')\big), \\
    L(a,b)\cup L(a,b')\cup L(a',b)\cup L(a',b') & =0, \\
                L(a,b)=L(a,b')=L(a',b)=L(a',b') & =0.
\end{align*}
Moreover the following are equivalent:
\begin{align*}
                                         c(a,b) & =1, \\
\Min U\big(L(a,b),L(a,b'),L(a',b),L(a',b')\big) & =1, \\
     U\big(L(a,b),L(a,b'),L(a',b),L(a',b')\big) & =1.
\end{align*}
\end{proof}

The mutual relationship between the compatibility relation and the commutator is expressed in the following result.

\begin{corollary}
If $(P,\leq,{}',0,1)$ is a horizontal sum of Boolean posets and $a,b\in P$ then $a\mathrel{{\rm C}}b$ if and only if $c(a,b)=1$.
\end{corollary}

\begin{proof}
This follows from Theorems~\ref{th1} and \ref{th2}.
\end{proof}

For the next theorem we extend the notion of the commutator from elements to subsets. For a generalized orthomodular poset $(P,\leq,{}',0,1)$ and subsets $A$ and $B$ of $P$ we define
\[
c(A,B):=\bigcup_{a\in A,b\in B}c(a,b).
\]

\begin{corollary}
The class of generalized orthomodular posets that are horizontal sums of Boolean posets satisfies the identity $c(c(x,y),z)\approx1$.
\end{corollary}

\begin{proof}
We have $c(x,y)\in\{0,1\}$ according to Theorem~\ref{th2} and $c(0,z)\approx c(1,z)\approx1$ according to Lemma~\ref{lem3}.
\end{proof}

Let $(P,\leq,{}',0,1)$ be a generalized orthomodular poset and $A\subseteq P$. Then we put $A':=\{x'\mid x\in A\}$ and define
\[
t(x,y,z):=\Min U\bigg(L\Big(\big(c(x,y)\big)',x\Big),L\big(c(x,y),z\big)\bigg)
\]
for all $x,y,z\in P$.

The next theorem shows that $t$ behaves on horizontal sums of Boolean posets similarly as the ternary discriminator.

\begin{theorem}
Let $(P,\leq,{}',0,1)$ be a generalized orthomodular poset that is a horizontal sum of Boolean posets and $a,b,c\in P$. Then
\[
t(a,b,c)=\left\{
\begin{array}{ll}
c & \text{if }a\mathrel{{\rm C}}b \\
a & \text{otherwise}.
\end{array}
\right.
\]
\end{theorem}

\begin{proof}
If $a\mathrel{{\rm C}}b$ then according to Theorem~\ref{th2} we have $c(a,b)=1$ and hence
\begin{align*}
t(a,b,c) & =\Min U\bigg(L\Big(\big(c(a,b)\big)',a\Big),L\big(c(a,b),c\big)\bigg)=\Min U\big(L(0,a),L(1,c)\big)= \\
         & =\Min U(0,c)=c.
\end{align*}
Otherwise, according to Theorem~\ref{th2} we have $c(a,b)=0$ and hence
\begin{align*}
t(a,b,c) & =\Min U\bigg(L\Big(\big(c(a,b)\big)',a\Big),L\big(c(a,b),c\big)\bigg)=\Min U\big(L(1,a),L(0,c)\big)= \\
         & =\Min U(a,0)=a.
\end{align*}
\end{proof}

Authors' addresses:

Ivan Chajda \\
Palack\'y University Olomouc \\
Faculty of Science \\
Department of Algebra and Geometry \\
17.\ listopadu 12 \\
771 46 Olomouc \\
Czech Republic \\
ivan.chajda@upol.cz

Miroslav Kola\v r\'ik \\
Palack\'y University Olomouc \\
Faculty of Science \\
Department of Computer Science \\
17.\ listopadu 12 \\
771 46 Olomouc \\
Czech Republic \\
miroslav.kolarik@upol.cz

Helmut L\"anger \\
TU Wien \\
Faculty of Mathematics and Geoinformation \\
Institute of Discrete Mathematics and Geometry \\
Wiedner Hauptstra\ss e 8-10 \\
1040 Vienna \\
Austria, and \\
Palack\'y University Olomouc \\
Faculty of Science \\
Department of Algebra and Geometry \\
17.\ listopadu 12 \\
771 46 Olomouc \\
Czech Republic \\
helmut.laenger@tuwien.ac.at
\end{document}